\documentclass{article}
\usepackage[utf8x]{inputenc}
\usepackage{amssymb}
\usepackage{amsmath}
\usepackage{mathrsfs}
\usepackage{bm}
\usepackage{amsthm}
\usepackage{enumitem}
\usepackage{color}
\usepackage{graphicx}
\usepackage{bm}
\usepackage{ulem}
\usepackage{upgreek}

\usepackage{accents}

\newcommand{\ci}[1]{\mathscr{#1}}
\newcommand{\g}[1]{\mathfrak{#1}}

\newcommand{\alfa}{\alpha}
\newcommand{\R}{\mathbf{R}}
\newcommand{\C}{\mathbf{C}}

\renewcommand{\H}{\mathbf{H}}

\newcommand{\bra}{\left\langle}
\newcommand{\ket}{\right\rangle}

\renewcommand{\phi}{\varphi}

\newcommand{\mi}{\mu}

\DeclareMathOperator{\End}{End}

\newtheorem{proposizione}{Proposition}[section]
\newtheorem{teorema}[proposizione]{Theorem}
\newtheorem{lemma}[proposizione]{Lemma}

\newtheorem*{lemma*}{Lemma}

\theoremstyle{definition}
\newtheorem{definizione}[proposizione]{Definition}

\theoremstyle{remark}
\newtheorem{osservazione}[proposizione]{Remark}

\title{Geometric structures arising from the deformation of groups of Heisenberg type}
\author{Claudio Afeltra}
\date{}

\begin{document}

\maketitle

\begin{abstract}
 Motivated by the desire of finding a geometric interpretation to the Yamabe equation on groups of Heisenberg type, we define a geometric structure on manifolds modelled locally on these groups, which we call contact structure of Heisenberg type. In the case of the Heisenberg group is equivalent to contact Riemannian manifolds.
 We define a natural connection on these structures, we compute the formula for the conformal change of scalar curvature, and introduce the Yamabe problem for these manifods.
\end{abstract}

\section{Introduction}
Analysis on Lie groups plays an important role in modern analysis, allowing to better understand which of the several properties of $\R^n$ are required for specific results, and having applications to differential geometry, partial differential equations, physics and other areas of mathematics and natural sciences.

Most results on Lie groups hold for specific subclasses of Lie groups.
A very studied family of Lie groups is the one of Carnot groups, namely simply connected groups $G$ whose Lie algebra $\g{g}$ has a stratification $\g{g}=\sum_{i=1}^kV_k$ with $[V_1,V_i]=V_{i+1}$ for all $i$ ($k$ being called the step of $G$).
The most important and studied example of Carnot group is the Heisenberg group $\H^n$, which is a step two Carnot group (the smallest possible nontrivial step) with a second step of dimension one, properties that make its study particularly convenient and in many cases allow for explicit computations not possible in the general case.

In \cite{K} Kaplan introduced a class of step two Carnot groups called groups of Heisenberg type which share many of properties of $\H^n$, and on which many of the computations on $\H^n$ can be repeated or adapted (see \cite{CDKR} Chapter 18 in \cite{BLU}).

One of these properties pertains to the equation
\begin{equation}\label{Equazione}
 -\Delta u = u^{\frac{Q+2}{Q-2}}, \;\;\; u>0
\end{equation}
where $\Delta$ is the sublaplacian and $Q$ is the homogeneous dimension (see Section \ref{SezionePreliminare} for the definition of such concepts).
Equation \eqref{Equazione} in the case of $\H^n$ is very important in differential geometry because it is equivalent to find the contact forms on $\H^n$ solving the CR Yamabe problem, that is the problem of findind a contact form on a CR manifold whose Webster curvature is constant (see \cite{JL1} or \cite{DT}).
On $\H^n$ this equation is known to have the solution
$$U(x,t) = C_n\left(\frac{1}{(1+|x|^2)^2 + 16|t|^2}\right)^{\frac{Q-2}{4}}$$
with $(x,t)\in\R^{2n}\times\R=\H^n$. Geometrically this solution corresponds to the standard contact form on $S^{2n-1}\subset\C^n$ pulled back to $\H^n$ through the Cayley transform (a transformation analogous to the stereographic projection).
$U$ and its translation and dilations are the unique solutions to Equation \eqref{Equazione} in $L^{\frac{2Q}{Q-2}}$ (see \cite{JL2}). It is an important open problem to find whether the same holds without integrability hypotheses.

Regarding groups of Heisenberg type, Garofalo and Vassilev discovered that Equation \eqref{Equazione} has a solution with the same form of the one on $\H^n$,
$$U(x,t) = C_G\left(\frac{1}{(1+|x|^2)^2 + 16|t|^2}\right)^{\frac{Q-2}{4}}$$
with $(x,t)\in\R^{2n}\times\R^k\simeq G$ (see \cite{GV}).
Later Yang (see \cite{Y}) proved that the classification result of Folland and Stein for $\H^n$ can be generalized to groups of Heisenberg type, that is that $U$ and its translation and dilations are the unique solutions to Equation \eqref{Equazione} in $L^{\frac{2Q}{Q-2}}$.

Despite all these similarities between the study of Equation \eqref{Equazione} on $\H^n$ and on groups of Heisenberg type, until now on the more general situation there is not a geometric interpretation of this equation analogous to the one on $\H^n$; that is, there is not a geometric structure allowing notions of scalar curvature and conformal change, such that the formula for the conformal change of the scalar curvature on $G$ is equivalent to Equation \eqref{Equazione}.

Our aim in this work is to find such a geometric interpretation.
For this purpose, we introduce a geometric structure which we call contact structure of Heisenberg type, locally modelled on a group of Heisenberg type. It consists of a subbundle $\g{H}$ of $TM$ with a subriemannian metric $g$ and of a vector space of forms $\ci{V}$ such that, associating to every $\theta\in\ci{V}$ the tensor $J_{\theta}$ defined by $g(X,J_{\theta}Y)=d\theta(X,Y)$, the association $\theta\mapsto J_{\theta}$ is a Clifford algebra, and furthermore a condition allowing the definition of Reeb vector fields analogous to the ones of contact geometry is satisfied (see Section \ref{SezioneDefinizione} for the complete definition).
In the case $G=\H^n$, our notion is equivalent to contact Riemannian manifolds.
This structure admit a natural notion of conformal change, as shown in Section \ref{SezioneDefinizione}.

In Section \ref{SezioneConnessione} we define a connection on contact manifolds of Heisenberg type. In the case of contact Riemannian manifolds our connection does not coincide with the Tanno connection, but with the Hermitian Tanno connection introduced by Nagase in \cite{N}. In the case of CR manifolds they both coincide with the Tanaka-Webster connection.

Finally in Section \ref{SezioneFormula} we define the scalar curvature in a manner analogous to Riemannian geometry, by contracting twice the curvature tensor, and we compute the formula for the conformal change of the scalar curvature, showing that up to an irrelevant constant it is equivalent to Equation \eqref{Equazione}.
We also introduce the Yamabe problem in this context, and notice that like in the Riemannian and in CR case, we can define a Yamabe constant $\ci{Y}(M)$ such that $\ci{Y}(M)\le\ci\ci{Y}(G)$, and such that if the strict inequality holds for a compact manifold $G$, then the Yamabe problem has solution.
Unlike the Riemannian and CR case, we are not able to find in general a compact manifold $M$ for which, analogously to $S^n$ in Riemannian geometry and $S^{2n-1}\subset\C^n$ in CR geometry, $\ci{Y}(M)=\ci\ci{Y}(G)$ (and we do not know whether it exists), but only for a subclass called groups of Iwasawa type. These groups admit a notion of spherical inversion which allows us to build a contact manifold of Heisenberg type diffeomorphic to $S^{2n+k}$ by glueing two copies of $G$ (see Section \ref{SezioneSfera}).

We point out that the idea of defining a structure on a manifold locally modelled on groups of Heisenberg type is not new, but in \cite{BGRV} the authors defined a structure which they called H-type foliation. Despite the similar goal, their structure does not coincide with ours and in particular it does not admit a notion of conformal change, which is our main motivation; the authors there were motivated by the generalization of other properties of groups of Heisenberg type to a non-flat case.

There are also connections between our work and the one of Biquard in \cite{B}, where he defined a geometric structure called quaternionic contact manifold.

\section{Groups of Heisenberg type}\label{SezionePreliminare}
In this section we will recall the definition and some known facts about groups of Heisenberg type.

Let $\g{g}$ be a real Lie algebra with a scalar product, $\g{z}$ be the center thereof, and $\g{v}=\g{z}^{\perp}$ the orthogonal of the center. Suppose that $[\g{v},\g{v}]\subset\g{z}$.
We define $J:\g{z}\to\End(\g{v})$ as
$$\bra J_TX,Y\ket = \bra T, [X,Y]\ket$$
for $T\in\g{z}$, $X,Y\in\g{v}$.
Then $\g{g}$ is said to be a Lie algebra of Heisenberg type if $J_T$ is an isometry for every $T\in\g{z}$ with $|T|=1$.
Equivalently, $\g{g}$ is a Lie algebra of Heisenberg type if for every $T\in\g{z}$
\begin{equation}\label{Antisimm}
 J_T^* = -J_T
\end{equation}
and
\begin{equation}\label{RelazioneCliffQ}
 J_T^2 = -|T|^2I.
\end{equation}
Condition \eqref{RelazioneCliffQ}, by polarization, and applying \eqref{Antisimm}, is equivalent to
$$ J_{T_1}J_{T_2} + J_{T_2}J_{T_1} = -2\bra T_1,T_2\ket I$$
for $T_1,T_2\in\g{z}$.

A Lie group $G$ with a scalar product on its Lie algebra $\g{g}$ is said a group of Heisenberg type if it is simply connected and $\g{g}$ is of Heisenberg type.

Therefore groups of Heisenberg type are a subclass of step two Carnot groups.

We recall that, given two real vector space $V$ and $W$ endowed with a scalar product, a homomorphism $J:V\to\End(W)$ is called a Clifford module if $J_v$ is antisymmetric for every $v$ and the Clifford relation
$$J_vJ_u+J_uJ_v = -2\bra v,u\ket I$$
holds for every $v,u\in V$.

Therefore in a Lie algebra $\g{g}=\g{v}\oplus\g{z}$ of Heisenberg type $\g{v}$ is a Clifford module over $\g{z}$. This allows to apply the theory of Clifford modules.
Vice versa, given a Clifford module $W$ over $V$, the space $W\times V$ with the bracket defined by $g([(X_1,T_1),(X_2,T_2)],S) = g(J_SX_1,X_2)$ for $X_1,X_2\in W$, $T_1,T_2,S\in V$ is a Lie algebra of Heiseneberg type.

Let us call $\dim(\g{v})=2n$ and $\dim(\g{z})=k$ (\eqref{RelazioneCliffQ} obviously implies that $\dim(\g{v})$ is even).
Then $\g{g}$, being a stratified Lie algebra, has the family of automorphisms
$$\delta_{\lambda}(X,T) = (\lambda X,\lambda^2T)$$
for $\lambda\in(0,\infty)$, and it holds that
$$(\delta_{\lambda}\mi) = \lambda^Q\mi$$
where $\mi$ is a Haar measure and $Q=2n+2k$ is called the homogeneous dimension of $\g{g}$.

We define the sublaplacian of $G$ as the differential operator
$$\Delta = \sum_{\alfa=1}^{2n}X_{\alfa}^2$$
where $X_1,\ldots,X_{2n}$ is an orthonormal basis of $\g{g}$ (it is elementary to verify that the definition does not depend on the basis).

The equation
$$-\Delta u = u^{\frac{Q+2}{Q-2}} ,\;\;\; u>0$$
has the solution
\begin{equation}\label{SoluzioneGarofaloLanconelli}
 U(x,t) = C_G\left(\frac{1}{(1+|x|^2)^2 + 16|t|^2}\right)^{\frac{Q-2}{4}}
\end{equation}
found by Garofalo and Vassilev in \cite{GV}, and which, thanks to a result of Yang in \cite{Y}, up to left translation and dilation, is the only solution in $L^{\frac{2Q}{Q-2}}$.

On groups of Heisenberg type a spherical inversion analogous to the one in $\R^n$ and on the Heisenberg group has been introduced in \cite{CDKR}:
\begin{equation}\label{DefinizioneInversione}
 \sigma(X,T) = \left(\left(-|X|^2I + 4J_T\right)^{-1}X , -\frac{T}{|X|^4+16|T|^2}\right)
\end{equation}
for $(X,T)\in\g{g}\setminus\{0\}$, and carried to $G$ by exponentiation.
$\sigma$ has meaningful properties only on a subclass of the groups of Heisenberg type, as the following result shows (see Theorems 4.2 and 5.1 in \cite{CDKR}).

\begin{teorema}\label{TeoremaIwasawa}
 The following are equivalent.
 \begin{itemize}
  \item $G$ is the nilpotent part of the Iwasawa decomposition of a rank one simple group.
  \item for every $X\in\g{v}$ and $T_1,T_2\in\g{z}$ there exist $T\in\g{z}$ such that $J_TX=J_{T_1}J_{T_2}X$.
  \item $\sigma$ preserves the horizontal distribution.
 \end{itemize}
\end{teorema}

A group verifying one of the equivalent statements of Theorem \ref{TeoremaIwasawa} is called a group of Iwasawa type.

\section{Contact structures of Heisenberg type}\label{SezioneDefinizione}
Let $M$ be a manifold, $\g{H}$ a $2n$-dimensional subbundle of $TM$ with a metric $g$, and $\ci{V}$ a $k$-dimensional vector space of differential forms whose restriction to $\g{H}$ is null and such that the induced forms on $TM/\g{H}$ constitute the whole dual thereof (which implies that $\dim(M)=2n+k$), endowed with a scalar product $\bra\cdot,\cdot\ket$.

We say that $(M,\g{H},g,\ci{V})$ is a contact manifold of Heisenberg type if:
\begin{itemize}
 \item for every $\theta\in\ci{V}$ the operator on $\g{H}$ defined by $g(X,J_{\theta}Y)=d\theta(X,Y)$ is such that the map $\ci{V}\to\ci{L}(\g{H}_x)$ given by $\theta\mapsto J_{\theta}$ is a Clifford algebra structure for every $x\in M$;
 \item calling $T_{\theta}$ the vector field defined by $i_{T_{\theta}}d\theta=0$ and $\psi(T_{\theta})=\bra\psi,\theta\ket$ for every $\psi\in\ci{V}$, it holds that the map $\theta\mapsto T_{\theta}$ is linear (and its image is called $\g{T}$, Reeb subspace).
\end{itemize}

We extend $g$ to the whole $TM$ imposing that $\g{H}$ and $\g{Z}$ be orthogonal and that the extension restricted to $\g{T}$ coincide with the product dual to $\bra\cdot,\cdot\ket$.

\begin{osservazione}
 An equivalent definition is defining a form $\bm{\theta}$ with values on a vector space $V$ with a scalar product, such that $\ker\bm{\theta}=\g{H}$ and the family of forms $\{\psi\circ\bm{\theta}\;|\;\psi\in V^*\}$ satisfies the properties of the original definition.
\end{osservazione}

\begin{osservazione}\label{OssRiemCont}
 If $\dim\ci{V}=1$ then a contact structure of Heisenberg type is equivalent to a Riemannian contact structure.
\end{osservazione}

\begin{osservazione}\label{Gruppi}
 If $G$ is a Lie group of Heisenberg type then thanks to the scalar product we can associate to every $T\in\g{z}$ a form which is zero on $\g{v}$, and
 $$d\theta_T(X,Y) = -\theta_T([X,Y]) = -\bra T,[X,Y]\ket =  -g(J_TX,Y) = g(X,J_TY)$$
 by definition, therefore it gives $G$ a structure of contact manifold of Heisenberg type.
\end{osservazione}

Given a contact manifold of Heisenberg type $M$, we call $G(M)$ the group of Heisenberg type whose Lie algebra is the one corresponding to the Clifford algebra $\ci{V}\to\ci{L}(\g{H}_x)$ in the definition (whose isomorphism type does not depend on $x$), and say that $M$ is locally modelled on $G(M)$.

\subsection{Conformal change}\label{SezioneCambioConforme}
Given $f$ smooth and nowhere zero, let $\widetilde{\ci{V}}=\{f\theta|\theta\in\ci{V}\}$ and $\widetilde{g}=fg$. Then
$$\widetilde{g}(X,J_{\theta}Y)= fg(X,J_{\theta}Y) = fd\theta(X,Y) = d(f\theta)(X,Y) = d\widetilde{\theta}(X,Y)$$
and therefore $J_{f\theta}=J_{\theta}$.
Therefore they form at every point a Clifford algebra structure.

Furthermore $\widetilde{T}_{f\theta}=\frac{1}{f}T_{\theta}+X_{\theta,f}$ with
$$0=i_{\frac{1}{f}T_{\theta}+X_{\theta,f}}(fd\theta +df\wedge\theta)|_{\g{H}} = fi_{X_{\theta,f}}d\theta|_{\g{H}} - \frac{|\theta|^2}{f}df|_{\g{H}} $$
therefore $X_{\theta,f}$ is determined uniquely by $i_{X_{\theta,f}}d\theta|_{\g{H}} = \frac{|\theta|^2}{f^2}df|_{\g{H}}$.

Since
$$\theta_{\psi}(T_{\theta_1}+T_{\theta_2}) = \bra\psi,\theta_1\ket + \bra\psi,\theta_2\ket = \bra\psi,\theta_1+\theta_2\ket$$
and obviously $T_{c\theta}= cT_{\theta}$, the linearity of $\theta\mapsto T_{\theta}$ is equivalent to
$$0 = i_{T_{\theta_1}+T_{\theta_2}}d(\theta_1+\theta_2)|_{\g{H}} = \left.\left(i_{T_{\theta_1}}d\theta_2 + i_{T_{\theta_2}}d\theta_1\right)\right|_{\g{H}}$$
and therefore, after the conformal change, calling $X=X_{\theta_1,f}$ and $Y=X_{\theta_2,f}$,
$$\left.\left(i_{\widetilde{T}_{\theta_1}}d\widetilde{\theta}_2 + i_{\widetilde{T}_{\theta_2}}d\widetilde{\theta}_1\right)\right|_{\g{H}}=$$
$$= \left.\left(i_{\frac{1}{f}T_{\theta_1}+X}(fd\theta_2 +df\wedge\theta_2) + i_{\frac{1}{f}T_{\theta_2}+Y}(fd\theta_1 +df\wedge\theta_1)\right)\right|_{\g{H}}=$$
$$= \left.\left(i_{T_{\theta_1}}d\theta_2 + fi_Xd\theta_2  + \frac{1}{f}T_{\theta_1}(f)\theta_2 - \frac{1}{f}\bra\theta_1,\theta_2\ket df + X(f)\theta_2 \right)\right|_{\g{H}}+$$
$$+\left.\left(i_{T_{\theta_2}}d\theta_1 + fi_Yd\theta_1  + \frac{1}{f}T_{\theta_2}(f)\theta_1 - \frac{1}{f}\bra\theta_1,\theta_2\ket df + Y(f)\theta_1\right)\right|_{\g{H}}=$$
$$= \left(i_{T_{\theta_1}}d\theta_2 +i_{T_{\theta_2}}d\theta_1 + fi_Xd\theta_2 + fi_Yd\theta_1 - \frac{2}{f}\bra\theta_1,\theta_2\ket df \right)|_{\g{H}}.$$
For every $Z\in\g{H}$ it holds that
$$fi_Xd\theta_2(Z) = fd\theta_2(X,Z) = fg(X,J_{\theta_2}Z) = -\frac{1}{|\theta_1|^2}fg(X,J_{\theta_1}^2J_{\theta_2}Z) =$$
$$= -\frac{1}{|\theta_1|^2}fd\theta_1(X,J_{\theta_1}J_{\theta_2}Z) = -\frac{1}{|\theta_1|^2}fi_Xd\theta_1(J_{\theta_1}J_{\theta_2}Z) = -\frac{1}{f}df(J_{\theta_1}J_{\theta_2}Z)$$
and analogously $fi_Yd\theta_1(Z) = -\frac{1}{f}df(J_{\theta_2}J_{\theta_1}Z)$, therefore
$$fi_Xd\theta_2(Z)+fi_Yd\theta_1(Z) = -\frac{1}{f}df(J_{\theta_1}J_{\theta_2}Z) -\frac{1}{f}df(J_{\theta_2}J_{\theta_1}Z) =$$
$$= -\frac{1}{f}df((J_{\theta_1}J_{\theta_2} +J_{\theta_2}J_{\theta_1})Z) = 2\bra\theta_1,\theta_2\ket \frac{1}{f}df(Z)$$
and therefore we have proved that the new structure satisfies the definition of contact manifold of Heisenberg type.

\subsection{The Iwasawa sphere}\label{SezioneSfera}
Thanks to the Garofalo-Lanconelli solution $U$ defined in formula \eqref{SoluzioneGarofaloLanconelli} we can define a manifold analogous to the sphere $S^n$ with its standard metric in Riemannian geometry and to the sphere $S^{2n-1}\subset\C^n$ with its standard contact form in CR geometry.

In fact, thanks to the computations in \cite{CDKR}, given a group of Heisenberg type $G$ with its standard contact structure of Heisenberg type $\ci{V}$ defined in Remark \ref{Gruppi}, if we perform a conformal change by defining $\widetilde{\ci{V}}= \{U^{\frac{4}{Q-2}}\theta\;|\;\theta\in\ci{V}\}$, then $\widetilde{\ci{V}}$ is left invariant by the spherical inversion $\sigma$ defined in formula \ref{DefinizioneInversione}, and therefore we can define a contact structure of Heisenberg type on $S^{2n+2k}$ by glueing two copies of $G$ along $G\setminus\{0\}$ through $\sigma$. We call $S^{2n+2k}$ with this structure the Iwasawa sphere $\bm{S}_G$. For example $\bm{S}_{H^n}$ is the sphere $S^{2n-1}\subset\C^n$ with its standard CR structure and contact form.

Thanks to Theorem \ref{TeoremaFormulaConforme}, we will show that the Iwasawa sphere has constant scalar curvature.

\section{A natural connection}\label{SezioneConnessione}
We recall that given a vector bundle $E$ over the manifold $M$ and a subbundle $V$ of $TM$, a partial connection on $E$ along $V$ is a linear map $\nabla:\Gamma(E)\to\Gamma(V*\times E)$ such that $\nabla s = df\oplus s + f\nabla s$. The restriction of an affine connection to $V$ is a partial connection, and viceversa if $TM=V\oplus W$ and $\nabla^V$ and $\nabla^W$ are partial connections along $V$ and $W$ respectively, there exists a unique affine connection whose restrictions along $V$ and $W$ coincide with $\nabla^V$ and $\nabla^W$.
If $E=V$ then the torsion tensor $T\in\Gamma(V^*\times V^*\times TM)$,
$$T(X,Y) = \nabla_XY-\nabla_YX -[X,Y]$$
is defined.

The aim of this section is to define an affine connection on $\g{H}$. Thanks to the above remarks, we can study the question of a partial connection along $\g{H}$ and a partial connection along $\g{T}$ separately.

\subsection{Partial connection along $\g{H}$}

\begin{lemma}\label{EsistenzaConnessione}
 There exists a connection on $\g{H}$ such that the space $\{J_{\theta}|\theta\in\ci{V}\}$ and $g$ are parallel.
\end{lemma}

\begin{proof}
 Since $\ci{V}\to\ci{L}(\g{H})$ is a Clifford algebra structure, it is a module over $\operatorname{Cliff}(\ci{V})$, and therefore by restriction a representation $\rho:\operatorname{Pin}(\ci{V})\to GL(\g{H})$ is induced.

 Let $\g{F}(\g{H})$ the proncipal bundle over $GL(\R^{2n})$ of the frames of $\g{H}$.
 Let $\widetilde{\rho}:\operatorname{Pin}(\R^k)\mapsto\operatorname{End}(\R^{2n})$ be a model representation isomorphic to $\rho$. For $\phi\in\operatorname{Pin}(\R^k)$ let $a_{ij}(g)$ be the coefficients of $\rho(g)$.
 Let $(X_1,\ldots,X_{2n})$ a smooth local frame for $\g{H}$ such that $\rho(\phi(g))X_i = a_{ij}(g)$.
 Then
 $$P=\{ (\sum_ia_{i1}(g)(X_1),\ldots,\sum_ia_{i,2n}(g)(X_{2n})) | g\in\operatorname{Pin}(\ci{V}_x)\}$$
 with the natural structure of principal $\operatorname{Pin}(\R^k)$-bundle is a reduction of the structure group of $\g{F}(\g{H})$ from $GL(\R^{2n})$ to $\rho(\operatorname{Pin}(\R^k))$.
 
 By the general theory of connections there exists a principal connection on $P$. Let $\nabla$ be the associated connection on $\g{H}$ through $\rho$. Since $\rho(\operatorname{Pin}(\ci{V}))$ is parallel, $\nabla_XJ_{\theta}$ must belong to the tangent space of $\rho(\operatorname{Pin}(\ci{V}))$ in $J_{\theta}$, that is $d\rho(T_{\theta}\operatorname{Pin}(\ci{V}))=J_\theta(\operatorname{span}(J_iJ_j))$ (where $J_i=J_{\theta^i}$ for some orthonormal basis $\theta^1,\ldots,\theta^k$ of $\ci{V}$).
 
 Therefore
 $$\nabla_XJ_i = J_i\sum_{j<k}a_{jk}J_jJ_k = \sum_{j\ne i}a_jJ_j + J_i\sum_{j<k,j,k\ne i}a_{jk}J_jJ_k $$
 It holds that
 $$0= -\nabla_X(I) = J_i\nabla_XJ_i + (\nabla_XJ_i)J_i =$$
 $$= J_i\sum_{j\ne i}a_jJ_j + J_iJ_i\sum_{j<k,j,k\ne i}a_{jk}J_jJ_k + \sum_{j\ne i}a_jJ_jJ_i + J_i\sum_{j<k,j,k\ne i}a_{jk}J_jJ_kJ_i =$$
 $$= -2\sum_{j<k,j,k\ne i}a_{jk}J_jJ_k $$
 So $\sum_{j<k,j,k\ne i}a_{jk}J_jJ_k=0$ and therefore $\nabla_XJ_i = \sum_{j\ne i}a_jJ_j$.
\end{proof}

\begin{lemma}\label{LemmaConnessione}
 Given a partial connection along $\g{H}$ satisfying the thesis of Lemma \ref{EsistenzaConnessione}, another connection $\widetilde{\nabla}_X=\nabla_X+a_X$ satisfies it if and only if
 $$a_X = \sum_{i<j}\alfa_{ij}(X)J_{\theta^i}J_{\theta^j} +b_X$$
 (where $\theta^1,\ldots,\theta^k$ is an orthonormal basis of $\ci{V}$)
 for some differential forms $\alfa_{ij}$ and for $b_X$ antisymmetric and satisfying $b_XJ_{\theta}=J_{\theta}b_X$ for every $X$ and $\theta$.
\end{lemma}

\begin{proof}
 It is a standard fact that $\widetilde{\nabla}g=0$ if and only if $a_X$ is antisymmetric.
 In order to leave $\{J_{\theta}|\theta\in\ci{V}\}$ parallel,
 $$((\widetilde{\nabla}_X-\nabla_X)J_{\theta})(Y) =$$
 $$= (\widetilde{\nabla}_X-\nabla_X)(J_{\theta}Y)-J_{\theta}((\widetilde{\nabla}_X-\nabla_X)Y)= a_X(J_{\theta}Y) - J_{\theta}(a_X(Y))$$
 must be equal to $J_{\psi}Y$ for some $\psi\in\ci{V}$. Differentiating $J_{\theta}^2=1$ we get that $(\nabla_XJ_{\theta})J_{\theta}+J_{\theta}\nabla_XJ_{\theta}=0$ and $(\widetilde{\nabla}_XJ_{\theta})J_{\theta}+J_{\theta}\widetilde{\nabla}_XJ_{\theta}=0$, therefore $J_{\psi}$ must anticommute with $J_{\theta}$, which implies that $\bra\psi,\theta\ket=0$.
 Calling $\psi=A(X)\psi$, in a basis
 $A(X)\theta^i=\sum_{j\ne i}\alfa_{ij}(X)\theta^j$.
 Differentiating the Clifford relation as before we get
 $$(\nabla_XJ_{\theta^i})J_{\theta^j}+ J_{\theta^i}(\nabla_XJ_{\theta^j})+ (\nabla_XJ_{\theta^j})J_{\theta^i} + J_{\theta^j}(\nabla_XJ_{\theta^i}) =0$$
 and
 $$(\widetilde{\nabla}_XJ_{\theta^i})J_{\theta^j}+ J_{\theta^i}(\widetilde{\nabla}_XJ_{\theta^j})+ (\widetilde{\nabla}_XJ_{\theta^j})J_{\theta^i} + J_{\theta^j}(\widetilde{\nabla}_XJ_{\theta^i}) =0,$$
 and subtracting we get
 $$J_{A(X){\theta^i}}J_{\theta^j} + J_{\theta^i}J_{A(X){\theta^j}} + J_{A(X){\theta^j}}J_{\theta^i} + J_{\theta^j}J_{A(X){\theta^i}} =0,$$
 and therefore
 $$0=\sum_{k\ne i}\alfa_{ik}(X)J_{\theta^k}J_{\theta^j} + \sum_{k\ne j}\alfa_{jk}J_{\theta^i}J_{\theta^k} + \sum_{k\ne j}\alfa_{jk}J_{\theta^k}J_{\theta^i} + \sum_{k\ne i}\alfa_{ik}(X)J_{\theta^j}J_{\theta^k} =$$
 $$\sum_{k\ne i}\alfa_{ik}(X)\left(J_{\theta^k}J_{\theta^j}+J_{\theta^j}J_{\theta^k}\right) + \sum_{k\ne j}\alfa_{jk}(X)\left(J_{\theta^i}J_{\theta^k}+J_{\theta^k}J_{\theta^i}\right)=$$
 $$=-2\alfa_{ij}(X) - 2\alfa_{ji}(X)$$
 which implies that $\alfa_{ij} + \alfa_{ji}=0$.
 
 The tensor $\sum_{i<j}\alfa_{ij}(X)J_{\theta^i}J_{\theta^j}$ satisfies
 $$\left(\sum_{i<j}\alfa_{ij}(X)J_{\theta^i}J_{\theta^j}\right)\circ J_{\theta^i} -J_{\theta^i}\circ\left(\sum_{i<j}\alfa_{ij}(X)J_{\theta^i}J_{\theta^j}\right) =$$
 $$= \sum_{j<k}\alfa_{jk}(X)\left(J_{\theta^j}J_{\theta^k}J_{\theta^i} - J_{\theta^i}J_{\theta^j}J_{\theta^k}\right) = $$
 $$= -2\sum_{j<i}\alfa_{ji}(X)J_{\theta^j} +2\sum_{j>i}\alfa_{ij}(X)J_{\theta^j} = 2\sum_{j\ne i}\alfa_{ij}(X)J_{\theta^j}$$
 and is antisymmetric, therefore $b_X:=a_X-\frac{1}{2}\sum_{i<j}\alfa_{ij}(X)J_{\theta^i}J_{\theta^j}$ is antisymmetric and satisfies $b_XJ_{\theta}=J_{\theta}b_X$ for every $\theta$.
\end{proof}

Let us define
\begin{equation}\label{DefinizioneQ}
 Q_{\theta}(X,Y) = J_{\theta}(\nabla_XJ_{\theta})Y -J_{\theta}(\nabla_YJ_{\theta})X + (\nabla_{J_{\theta}Y}J_{\theta})X - (\nabla_{J_{\theta}X}J_{\theta})Y
\end{equation}
which can easily be verified to be a tensor.

\begin{lemma}\label{LemmaTensoreQ}
 In the notation of Lemma \ref{LemmaConnessione} there exist a unique choice of forms $\alfa_{ij}$ such that $\widetilde{Q}_{\theta}=0$, while the addition of a tensor $b$ leaves the tensor $Q$ invariant.
\end{lemma}

\begin{proof}
 Since
 $$(\nabla_XJ_{\theta})(Y) = \nabla_X(J_{\theta}Y) - J_{\theta}(\nabla_XY)$$
 and analogously for $\widetilde{\nabla}$, we have
 \begin{equation}\label{DifferenzialeJ}
  (\nabla_X-\widetilde{\nabla}_X)J_{\theta} = a_XJ_{\theta} - J_{\theta}a_X,
 \end{equation}
 therefore
 $$\widetilde{Q}_{\theta^k}(X,Y) -Q_{\theta^k}(X,Y) =  $$
 $$ = J_{\theta^k}a_XJ_{\theta^k}Y +a_XY -J_{\theta^k}a_YJ_{\theta^k}X -a_YX +$$
 $$+ a_{J_{\theta^k}Y}J_{\theta^k}X -J_{\theta^k}a_{J_{\theta^k}Y}X - a_{J_{\theta^k}X}J_{\theta^k}Y + a_{J_{\theta^k}X}J_{\theta^k}Y=$$
 $$=\sum_{i<j}\alfa_{ij}(X)\left(J_{\theta^k}J_{\theta^i}J_{\theta^j}J_{\theta^k}Y + J_{\theta^i}J_{\theta^j}Y\right)+$$
 $$-\sum_{i<j}\alfa_{ij}(Y)\left(J_{\theta^k}J_{\theta^i}J_{\theta^j}J_{\theta^k}X + J_{\theta^i}J_{\theta^j}X\right)+$$
 $$+\sum_{i<j}\alfa_{ij}(J_{\theta^k}Y)\left(J_{\theta^i}J_{\theta^j}J_{\theta^k}X - J_{\theta^k}J_{\theta^i}J_{\theta^j}X \right) +$$
 $$-\sum_{i<j}\alfa_{ij}(J_{\theta^k}X)\left(J_{\theta^i}J_{\theta^j}J_{\theta^k}Y - J_{\theta^k}J_{\theta^i}J_{\theta^j}Y\right)=$$
 \begin{equation}\label{DifferenzaQ}
  =2\sum_{i\ne k}\left(\alfa_{ik}(X)J_{\theta^i}J_{\theta^k}Y-\alfa_{ik}(Y)J_{\theta^i}J_{\theta^k}X - \alfa_{ik}(J_{\theta^k}Y)J_{\theta^i}X + \alfa_{ik}(J_{\theta^k}X)J_{\theta^i}Y\right).
 \end{equation}
 Calling $\nabla_XJ_{\theta^k} = \sum_{i\ne k}\beta_{ki}(X)J_{\theta^i}$ we have that
 $$Q_{\theta^k}(X,Y) =$$
 $$= \sum_{i\ne k}\left( J_{\theta^k}\beta_{ki}(X)J_{\theta^i}Y -J_{\theta^k}\beta_{ki}(Y)J_{\theta^i}X + \beta_{ki}(J_{\theta^k}Y)J_{\theta^i}X - \beta_{ki}(J_{\theta^k}X)J_{\theta^i}Y \right).$$
 Therefore choosing $\alfa_{ki}=\frac{1}{2}\beta_{ki}$ we have that $\widetilde{Q}_{\theta^k}=0$.
 
 Now let us suppose that $Q_{\theta}=\widetilde{Q}=0$ for every $\theta$. Given $\ell\ne k$, thanks to formula \eqref{DifferenzaQ} it is easy to show that
 $$J_{\theta_{\ell}}\circ\left(Q_{\theta^k}-\widetilde{Q}_k\right) - \left(Q_{\theta^k}-\widetilde{Q}_k\right)\circ J_{\theta_{\ell}} = $$
 $$= 2\left(-\alfa_{\ell k}(X)J_{\theta^k}Y -\alfa_{\ell k}(Y)J_{\theta^k}X - \alfa_{\ell k}(J_{\theta^k}Y)X + \alfa_{\ell k}(J_{\theta^k}X)Y\right),$$
 and therefore
 $$\alfa_{ik}(X)J_{\theta^i}J_{\theta^k}Y-\alfa_{ik}(Y)J_{\theta^i}J_{\theta^k}X - \alfa_{ik}(J_{\theta^k}Y)J_{\theta^i}X + \alfa_{ik}(J_{\theta^k}X)J_{\theta^i}Y =0$$
 for every $i,k$. If $Y=J_{\theta^i}X$ then
 \begin{equation}\label{Somma4Vettori}
  \alfa_{ik}(X)J_{\theta^k}X-\alfa_{ik}(J_{\theta^i}X)J_{\theta^i}J_{\theta^k}X - \alfa_{ik}(J_{\theta^k}J_{\theta^i}X)J_{\theta^i}X - \alfa_{ik}(J_{\theta^k}X)X =0.
 \end{equation}
 Since $J_{\theta^i}$ and $J_{\theta^k}$ are antisymmetric $g(X,J_{\theta^i}X)=g(X,J_{\theta^k}X)=0$ and
 $$g(X,J_{\theta^i}J_{\theta^k}X) = -g(J_{\theta^i}X,J_{\theta^k}X) = g(J_{\theta^k}J_{\theta^i}X,X) = -g(J_{\theta^i}J_{\theta^k}X,X)$$
 and therefore $g(X,J_{\theta^i}J_{\theta^k}X)=g(J_{\theta^i}X,J_{\theta^k}X)=0$. Therefore formula \eqref{Somma4Vettori} is the sum of four linearly independent vectors, and from this we deduce that $\alfa_{ik}=0$ for every $i,k$ and therefore that $Q_{\theta}=\widetilde{Q}_{\theta}$ for every $\theta$.
\end{proof}

Now, in the notation of Lemma \ref{LemmaConnessione}, and supposing that $\alfa_{ij}=0$ thanks to Lemma \ref{LemmaTensoreQ}, the torsion of $\widetilde{\nabla}$ is
\begin{equation}\label{FormulaCambioTorsione}
 \widetilde{T}(X,Y) = T(X,Y) +(\widetilde{\nabla}_X-\nabla_X)Y -(\widetilde{\nabla}_Y-\nabla_Y)X= T(X,Y) +b_XY-b_YX.
\end{equation}
Let us call $\pi_{\g{H}}T$ the projection of the torsion on $\g{H}$ with respect to the direct sum with $\g{T}$.
Now, given the space of the tensors of type $(0,3)$ on $\g{H}$, $T^{(0,3)}(\g{H})$, let us consider the operator of antisymmetrization with respect to the first two entries.
Its restriction to the subspace of tensors antisymmetric with respect to the second and the third entry is injective thanks to the folowing lemma

\begin{lemma}
 If $T$ is a tensor such that $T(X,Y,Z)=-T(X,Z,Y)$ and\\ $T(X,Y,Z)=T(Y,X,Z)$ then $T=0$.
\end{lemma}

\begin{lemma}
 Considering the action of the permutation group $\g{S}_3$ on $\{T,-T\}$ given by $\sigma T(X_1,X_2,X_3)=T(X_{\sigma(1)},X_{\sigma(2)},X_{\sigma(3)})$, since transpositions of $\g{S}_3$ are conjugated, and by hypothesis $(12)$ is in the kernel, then also $(23)$ is, but then $T=0$.
 Alternatively
 $$T(X,Y,Z)=T(Y,X,Z)= -T(Y,Z,X) = -T(Z,Y,X) = T(Z,X,Y) =$$
 $$= T(X,Z,Y) = -T(X,Y,Z).$$
\end{lemma}

Now considering the subspace of $T^{(0,3)}(\g{H})$ such that they commute with $J_{\theta}$ for every $\theta$ with respect to the second and the third variable, let $\Sigma(\g{H})$ be the image of the antisymmetrization operator with respect to the first two entries.
Then, by formula \eqref{FormulaCambioTorsione} and the above reasoning, there exist a unique $b$ such that the torsion $\pi_{\g{H}}\widetilde{T}$ is orthogonal to $\Sigma(\g{H})$.
Therefore we get the following.

\begin{proposizione}\label{PropConnessioneH}
 There exists a unique partial connection on $\g{H}$ along $\g{H}$ such that $\{J_{\theta}|\theta\in\ci{V}\}$ and $g$ are parallel, $Q_{\theta}=0$ for every $\theta$ (where $Q_{\theta}$ is defined in formula \eqref{DefinizioneQ}) and such that $\pi_{\g{H}}T$ is orthogonal to $\Sigma(\g{H})$.
\end{proposizione}

\subsection{Partial connection along $\g{T}$}

\begin{lemma}\label{LemmaConnessioneReeb}
 There exists a partial connection along $\g{T}$ for $\g{H}$ such that $\{J_{\theta}|\theta\in\ci{V}\}$ and $g$ are parallel, and any other such connection is of the form
 $$\widetilde{\nabla}_T = \nabla_T + \sum_{i<j}\beta_{ij}(T)J_{\theta^i}J_{\theta^j} +\gamma_T$$
 with $\gamma_T$ antisymmetric and such that $\gamma_TJ_{\theta}=J_{\theta}\gamma_T$ for every $T$ and $\theta$.
\end{lemma}

\begin{proof}
 The proof is the same of Lemma \ref{LemmaConnessione}.
\end{proof}

\begin{lemma}\label{LemmaReeb2}
 In the notation of Lemma \ref{LemmaConnessioneReeb}, there exists a unique choice of $\beta_{ij}$ such that $\nabla_TJ_{\theta}=0$ for every $T$ and $\theta$, and the addition of a tensor $\gamma$ leaves $\nabla_TJ_{\theta}$ invariant.
\end{lemma}

\begin{proof}
 The proof is the same of Lemma \ref{LemmaTensoreQ}. 
\end{proof}

Now let us define the partial torsion
$$T^{\nabla}(T,X) = \nabla_TX-\pi_{\g{H}}[T,X].$$
Then, in the notation of Lemma \ref{LemmaConnessioneReeb}, and assuming that $\beta_{ij}=0$ thanks to Lemma \ref{LemmaReeb2},
hanks to Lemma \ref{LemmaReeb2},
\begin{equation}\label{FormulaCambioTorsioneParziale}
 T^{\widetilde{\nabla}}(T,X) = \nabla_TX + \gamma_T(X) -\pi_{\g{H}}[T,X] = T^{\nabla}(T,X) + \gamma_T(X).
\end{equation}
Let us call $\Xi(\g{H})$ the space of antisymmetric tensors commuting with $J_{\theta}$ for every $\theta$.
Then, thanks to formula \eqref{FormulaCambioTorsioneParziale}, there exists a unique $\gamma$ such that the operator $T^{\widetilde{\nabla}}(T,\cdot)$ is orthogonal to $\Xi(\g{H})$ for every $T$.
Therefore we get the following.

\begin{proposizione}\label{PropConnessioneT}
	There exists a unique partial connection on $\g{H}$ along $\g{T}$ such that $g$ is parallel, $\nabla_TJ_{\theta}=0$ for every $T$ and $\theta$, and such that $T^{\widetilde{\nabla}}(T,\cdot)$ is orthogonal to $\Xi(\g{H})$ for every $T$.
\end{proposizione}

Therefore we can now define a connection on $M$.

\begin{definizione}
 Given a contact manifold of Heisenberg type, we define $\nabla$ as the connection for which $\g{H}$ is parallel, whose reduction to $\g{H}$ restricts to the partial connection defined in Proposition \ref{PropConnessioneH} on $\g{H}$ and to the partial connection defined in Proposition \ref{PropConnessioneT} on $\g{T}$, and for which the Reeb vector fields are parallel.
\end{definizione}

\begin{osservazione}
 In the case of $G$ with the contact structure of Heisenberg type defined in Remark \ref{Gruppi}, our connection coincides with the connection for which left invariant vector fields are parallel.
\end{osservazione}

\begin{osservazione}
 In the case $G=\H^n$, which corresponds to contact Riemannian manifolds as seen in Remark \ref{OssRiemCont}, our connection coincides with the Hermitian Tanno connection defined by Nagase in \cite{N}, which in the case of CR manifolds coincides with the Tanaka-Webster connection.
\end{osservazione}

\section{Scalar curvature}\label{SezioneFormula}
Since we defined a connection on $\g{H}$ we get the curvature tensor
$$R(X,Y)Z = \nabla_X\nabla_YZ - \nabla_Y\nabla_XZ -\nabla_{[X,Y]}Z$$
for $X,Y\in TM$ and $Z\in\g{H}$.
By restricting $X$ and $Y$ to $\g{H}$, and contracting this tensor two times with $g$, we get a scalar curvature invariant $R$.
Explicitely, if $X_1,\ldots,X_{2n}$ is an orthonormal frame of $\g{H}$,
$$R = \sum_{\alfa,\beta=1}^{2n}g(R(X_{\alfa},X_{\beta})X_{\alfa},X_{\beta}). $$
Given a conformal change as defined in Subsection \ref{SezioneCambioConforme}, we want to compute the scalar curvature of the new connection.

First we will compute the change of the connection.

Let us consider the operator of antisymmetrization with respect with the first two variables from the space of operators antisymmetric with respect to the second and third variable commuting with $J_{\theta}$ for all $\theta$ to $\Sigma$, and call $\Theta$ its inverse.

In the following it is useful to notice that symmetry and antisymmetry do not depend on the conformal metric, and in the same way $\Sigma(\g{H})$, $\Xi(\g{H})$ and the associated projections.

\begin{proposizione}
 $$\widetilde{\nabla}_X = \nabla_X + \frac{Xf}{2f}Y + \frac{1}{f}c_X(Y,df)$$
 where
 $$c_X(Y,df) = \Theta\left(\pi_{\Sigma}\left((X,Y)\mapsto \frac{Yf}{2f}X -\frac{Xf}{2f}Y + \alfa(X,Y,df)\right)\right)$$
 and
 $$g(\alfa(X,Y,df),Z) = \sum_{j=1}^kg(X,J_{\theta^j}Y)df(J_{\theta^j}Z) . $$
\end{proposizione}

\begin{proof}
 Define $a_X = \widetilde{\nabla}_X-\nabla_X$. Since $\widetilde{g}=fg$, in order that $\widetilde{\nabla}\widetilde{g} = 0$
 $$0 = (\widetilde{\nabla}_X(fg))(Y,Z) = X(fg(Y,Z)) - fg(\widetilde{\nabla}_XY,Z) - fg(Y,\widetilde{\nabla}_XZ) =$$
 $$= (Xf)g(Y,Z) + fX(g(Y,Z)) - fg(\nabla_XY,Z) +$$
 $$- fg(a_XY,Z) - fg(Y,\nabla_XZ) - fg(Y,a_XZ) =$$
 $$= (Xf)g(Y,Z) + f(\nabla_Xg)(Y,Z) - fg(a_XY,Z) - fg(Y,a_XZ) =$$
 $$= (Xf)g(Y,Z)  - fg(a_XY,Z) - fg(Y,a_XZ) $$
 and therefore $a_XY=\frac{Xf}{2f}Y + b_XY$ with $b_X$ antisymmetric.
 By computations similar to the ones in Lemma \ref{LemmaConnessione}
 $$b_X = \sum_{i<j}\alfa_{ij}(X)J_{\theta^i}J_{\theta^j}+c_X$$
 with $c_X$ antisymmetric and commuting with $J_{\theta}$ for every $\theta$.
 By formula \eqref{DifferenzialeJ} it is easy to deduce that $\alfa_{ij}=0$, therefore
 $$ \widetilde{\pi_{\g{H}}}\widetilde{T}(X,Y)=$$
 \begin{equation}\label{CambioConformeTorsione}
   \widetilde{\pi_{\g{H}}}T(X,Y) + a_X(Y)-a_Y(X) = \widetilde{\pi_{\g{H}}}T(X,Y) + \frac{Xf}{2f}Y + c_XY -\frac{Yf}{2f}X - c_YX.
 \end{equation}
 But since $\g{T}$ is not preserved by conformal change, neither is $\pi_{\g{H}}$.
 Since
 $$\theta^j(T(X,Y)) = -\theta^j([X,Y]) = d\theta^j(X,Y)= g(X,J_{\theta}Y)$$
 we have that
 $$T(X,Y) = \sum_{j=1}^kg(X,J_{\theta^j}Y)T_j + (\pi_{\g{H}}T)(X,Y).$$
 We saw in Subsection \ref{SezioneCambioConforme} that
 $$\widetilde{T}^j = \frac{1}{f}T^j + X^j$$
 with $i_{X^j}d\theta^j|_{\g{H}} = \frac{1}{f^2}df|_{\g{H}}$, hence
 $$T(X,Y) = (\pi_{\g{H}}T)(X,Y) + \sum_{j=1}^kg(X,J_{\theta^j}Y)(f\widetilde{T}_j - fX_j)$$
 and taking the scalar product with $Z\in\g{H}$ we get
 $$g(\widetilde{\pi}_{\g{H}}T(X,Y),Z) = - f\sum_{j=1}^kg(X,J_{\theta^j}Y)g(X_j,Z) =$$
 $$= f\sum_{j=1}^kg(X,J_{\theta^j}Y)d\theta^j(X_j,J_{\theta^j}Z) = \frac{1}{f}\sum_{j=1}^kg(X,J_{\theta^j}Y)df(J_{\theta^j}Z) =$$
 $$= \frac{1}{f}\left(\sum_{j=1}^kg(X,J_{\theta^j}Y)J_{\theta^j}Z\right)\cdot \nabla^{\g{H}}f = g(\alfa(X,Y),Z)$$
 which together with formula \eqref{CambioConformeTorsione} implies the thesis.
\end{proof}

\begin{proposizione}
 $$\widetilde{\nabla}_{\widetilde{T}_{\theta}} =  \frac{1}{f}\nabla_{T_{\theta}} + \nabla_X + \frac{Xf}{2f}Y +\frac{T_{\theta}f}{2f^2}Y + \frac{1}{2f}\sum_{i=1}^k(\nabla_XJ_{\theta^i})J_{\theta^i}Y+ \frac{1}{f}\gamma_{T_{\theta}}Y $$
 where
 $$\gamma_{T_{\theta}}Y =\pi_{\Xi}\left(Y\mapsto -\frac{1}{2}\sum_{i=1}^k(\nabla_XJ_{\theta^i})J_{\theta^i}Y+\right.$$
 $$\left. - \nabla_YX - \pi_{\g{H}}T(X,Y)) +\frac{1}{f^2}\beta(Y,df) +\frac{1}{f^3}\sigma(Y,|df|^2) \right).$$
\end{proposizione}

\begin{proof}
 We have
 \begin{equation}\label{FormulaCambioConformeReeb}
  \widetilde{\nabla}_{\widetilde{T}_{\theta}} = \frac{1}{f}\widetilde{\nabla}_{T_{\theta}} + \widetilde{\nabla}_{X} = \frac{1}{f}\nabla_{T_{\theta}} + \frac{1}{f}\eta_T + \nabla_X + \frac{Xf}{2f}Y + \frac{1}{f}c_X(Y,df)
 \end{equation}
 Since $\widetilde{\nabla}_{X}\widetilde{g}=0$, the fact that $\widetilde{\nabla}_{\widetilde{T}_{\theta}}\widetilde{g}=0$ is equivalent to
 $$0= (\widetilde{\nabla}_{T_{\theta}}\widetilde{g})(X,Y) =$$
 $$= (\widetilde{\nabla}_{T_{\theta}}(fg))(X,Y) = T_{\theta}(fg(X,Y)) - fg(\widetilde{\nabla}_{T_{\theta}}X,Y) - fg(X,\widetilde{\nabla}_{T_{\theta}}Y) = $$
 $$ =   (T_{\theta}f)g(X,Y)  + f(\nabla_{T_{\theta}}g)(X,Y) - fg(\eta_{T_{\theta}}X,Y) - fg(X,\gamma_{T_{\theta}}Y) = $$
 $$ =   (T_{\theta}f)g(X,Y) - fg(\eta_{T_{\theta}}X,Y) - fg(X,\eta_{T_{\theta}}Y) $$
 therefore $\eta_{T_{\theta}} = \frac{T_{\theta}f}{2f}I + E_{T_{\theta}}$ with $E_{T_{\theta}}$ antisymmetric.
 Using formula \eqref{FormulaCambioConformeReeb},
 $$\widetilde{\nabla}_{\widetilde{T}_{\theta}}J = \nabla_XJ + E_{T_{\theta}}J - JE_{T_{\theta}}$$
 which implies that
 $$E_{T_{\theta}} = \frac{1}{2}\sum_{i=1}^k(\nabla_XJ_{\theta^i})J_{\theta^i}+\gamma_{T_{\theta}}$$
 with $\gamma_{T_{\theta}}$ antisymmetric and commuting with $J_{\theta}$ for every $\theta$.
 
 Now the partial torsion after the conformal change is
 $$T^{\widetilde{\nabla}}(\widetilde{T}_{\theta},Y) = \widetilde{\nabla}_{\widetilde{T}_{\theta}}Y - \widetilde{\pi}_{\g{H}}[\widetilde{T}_{\theta},Y] =$$
 $$= \frac{1}{f}\nabla_{T_{\theta}}Y + \frac{T_{\theta}f}{2f^2}Y + \frac{1}{2}\sum_{i=1}^k(\nabla_XJ_{\theta^i})J_{\theta^i}+$$
 $$+\gamma_{T_{\theta}}Y + \nabla_XY + \frac{Xf}{2f}Y + \frac{1}{f}c_X(Y,df) - \widetilde{\pi}_{\g{H}}[\widetilde{T}_{\theta},Y] $$
 It holds that
 $$\widetilde{\pi}_{\g{H}}[\widetilde{T}_{\theta},Y] = [\widetilde{T}_{\theta},Y] - \sum_{i=1}^k\widetilde{\theta}^i([\widetilde{T}_{\theta},Y])\widetilde{T}_{\theta^i} =$$
 $$= [\widetilde{T}_{\theta},Y] + f\sum_{i=1}^kd\theta^i(\widetilde{T}_{\theta},Y)\widetilde{T}_{\theta^i} =$$
 $$= [\frac{1}{f}T_{\theta}+X,Y] + f\sum_{i=1}^kd\theta^i(\frac{1}{f}T_{\theta}+X,Y)\left(\frac{1}{f}T_{\theta^i} + X_i \right)=$$
 $$= \frac{1}{f}[T_{\theta},Y] - Y\left(\frac{1}{f}\right)T_{\theta} + [X,Y] + \sum_{i=1}^kd\theta^i(T_{\theta},Y)\left(\frac{1}{f}T_{\theta^i}+ X_i \right)+$$
 $$ + f\sum_{i=1}^kd\theta^i(X,Y)\left(\frac{1}{f}T_{\theta^i} + X_i \right) =$$
 $$= \frac{1}{f}[T_{\theta},Y] - Y\left(\frac{1}{f}\right)T_{\theta} + [X,Y] + \frac{1}{f}\sum_{i=1}^kd\theta^i(T_{\theta},Y)T_{\theta^i} +$$
 $$+ \sum_{i=1}^kd\theta^i(T_{\theta},Y)X_i + \sum_{i=1}^kd\theta^i(X,Y)T_{\theta^i}+ f\sum_{i=1}^kd\theta^i(X,Y)X_i=$$
 $$= \frac{1}{f}\pi_{\g{H}}[T_{\theta},Y] - Y\left(\frac{1}{f}\right)T_{\theta} + \pi_{\g{H}}[X,Y] +$$
 $$+ \sum_{i=1}^kd\theta^i(T_{\theta},Y)X_i + f\sum_{i=1}^kd\theta^i(X,Y)X_i=$$
 $$= \frac{1}{f}\pi_{\g{H}}[T_{\theta},Y] + \pi_{\g{H}}[X,Y] + \frac{1}{f^2}\beta(Y,df) + \frac{1}{f^3}\sigma(Y,|df|^2).$$
 So
 $$T^{\widetilde{\nabla}}(\widetilde{T}_{\theta},Y) =$$
 $$= \frac{1}{f}\nabla_{T_{\theta}}Y + \frac{T_{\theta}f}{2f^2}Y + \frac{1}{2}\sum_{i=1}^k(\nabla_XJ_{\theta^i})J_{\theta^i}+\gamma_{T_{\theta}}Y + \nabla_XY + \frac{Xf}{2f}Y + \frac{1}{f}c_X(Y,df)+$$
 $$ - \frac{1}{f}\pi_{\g{H}}[T_{\theta},Y] - \pi_{\g{H}}[X,Y] - \frac{1}{f^2}\beta(Y,df) -\frac{1}{f^3}\sigma(Y,|df|^2) =$$
 $$= \frac{1}{f}T^{\nabla}(T_{\theta},Y) + \frac{T_{\theta}f}{2f^2}Y + \frac{1}{2}\sum_{i=1}^k(\nabla_XJ_{\theta^i})J_{\theta^i}+\gamma_{T_{\theta}}Y + \nabla_YX +\pi_{\g{H}}T(X,Y) +$$
 $$+ \frac{Xf}{2f}Y + \frac{1}{f}c_X(Y,df) - \frac{1}{f^2}\beta(Y,df) -\frac{1}{f^3}\sigma(Y,|df|^2)$$
 Therefore
 $$\gamma_{T_{\theta}}Y =\pi_{\Xi}\left(Y\mapsto -\frac{1}{2}\sum_{i=1}^k(\nabla_XJ_{\theta^i})J_{\theta^i}Y - \nabla_YX - \pi_{\g{H}}T(X,Y))+\right.$$
 $$\left. +\frac{1}{f^2}\beta(Y,df) +\frac{1}{f^3}\sigma(Y,|df|^2) \right)-\frac{1}{f}c_X(Y,df).$$
\end{proof}

In order to state our formula, we define the sublaplacian $\Delta u$ of a function $u$ on $M$ as the trace of the restriction of $\nabla^2f$ on $\g{H}$, that is, given an orthonormal frame $X_1,\ldots,X_{2n}$ of $\g{H}$,
$$\Delta u = \sum_{\alfa=1}^{2n}\nabla^2_{X_{\alfa},X_{\alfa}}u.$$

\begin{teorema}\label{TeoremaFormulaConforme}
 There exist a constant $C_G$ such that, for any contact manifold of Heisenberg type modelled on $G$, calling the conformal factor $f=u^{\frac{4}{Q-2}}$ it holds that
 $$\widetilde{K} = u^{-\frac{Q+2}{Q-2}}\left(-C\Delta u + Ku \right).$$
\end{teorema}

\begin{proof}
 Let $\overline{x}\in M$, and let $X_1,\ldots,X_{2n}$ an orthonormal frame for $g$ such that $\nabla_{X_{\alfa}}X_{\beta}=0$ at $\overline{x}$. Then
 $$[X_{\alfa},X_{\beta}](\overline{x}) = -T(X_{\alfa},X_{\beta})(\overline{x})$$
 and therefore
 $$\widetilde{K} =\widetilde{g}(\widetilde{R}(X_{\alfa},X_{\beta})X_{\alfa},X_{\beta}) = g(\widetilde{\nabla}_{X_{\alfa}}\widetilde{\nabla}_{X_{\beta}}X_{\alfa}-\widetilde{\nabla}_{X_{\beta}}\widetilde{\nabla}_{X_{\alfa}}X_{\alfa}-\widetilde{\nabla}_{[X_{\alfa},X_{\beta}]}X_{\alfa},X_{\beta}) =$$
 $$= fg\left(\nabla_{X_{\alfa}}\nabla_{X_{\beta}}X_{\alfa} + \frac{X_{\alfa}X_{\beta}f}{2f}X_{\alfa} - \frac{X_{\alfa}fX_{\beta}f}{4f^2}X_{\alfa} + \frac{X_{\beta}f}{2f}\nabla_{X_{\alfa}}X_{\alfa} +  \frac{X_{\alfa}f}{2f}\nabla_{X_{\beta}}X_{\alfa}+\right.$$
 $$ + c_{\nabla_{X_{\alfa}}X_{\beta}}X_{\alfa} + 2c_{X_{\alfa}}(\nabla_{X_{\beta}}X_{\alfa})+ \frac{X_{\alfa}f}{2f}c_{X_{\beta}}X_{\alfa} + \frac{X_{\beta}f}{2f}c_{X_{\alfa}}X_{\alfa} + c_{X_{\alfa}}(c_{X_{\beta}}X_{\alfa}) +$$
 $$-\left( \nabla_{X_{\beta}}\nabla_{X_{\alfa}}X_{\alfa} + \frac{X_{\beta}X_{\alfa}f}{2f}X_{\alfa} - \frac{X_{\beta}fX_{\alfa}f}{4f^2}X_{\alfa} + \frac{X_{\alfa}f}{2f}\nabla_{X_{\beta}}X_{\alfa} +  \frac{X_{\beta}f}{2f}\nabla_{X_{\alfa}}X_{\alfa} +\right.$$
 $$\left.+ c_{\nabla_{X_{\beta}}X_{\alfa}}X_{\alfa} + 2c_{X_{\beta}}(\nabla_{X_{\alfa}}X_{\alfa}) + \frac{X_{\beta}f}{2f}c_{X_{\alfa}}X_{\alfa} + \frac{X_{\alfa}f}{2f}c_{X_{\beta}}X_{\alfa} + c_{X_{\beta}}(c_{X_{\alfa}}X_{\alfa})\right) +$$
 $$+\nabla_{(\pi_{\g{H}}T)(X_{\alfa},X_{\beta})}X_{\alfa} + \frac{(\pi_{\g{H}}T)(X_{\alfa},X_{\beta})f}{2f}X_{\alfa} + c_{(\pi_{\g{H}}T)(X_{\alfa},X_{\beta})}X_{\alfa} +$$
 $$ +\nabla_{(\pi_{\g{T}}T)(X_{\alfa},X_{\beta})}X_{\alfa} + \frac{(\pi_{\g{T}}T)(X_{\alfa},X_{\beta})f}{2f}X_{\alfa} +$$
 $$\left.+\frac{1}{2f}\sum_{i=1}^k(\nabla_{(\pi_{\g{T}}T)(X_{\alfa},X_{\beta})}J_{\theta^i})J_{\theta^i}X_{\alfa} +  \frac{1}{f}\gamma_{(\pi_{\g{T}}T)(X_{\alfa},X_{\beta})}X_{\alfa}, X_{\beta}\right)=$$
 $$= fg\left(R(X_{\alfa},X_{\beta})X_{\alfa}  + \frac{X_{\alfa}X_{\beta}f}{2f}X_{\alfa} +\right.$$
 $$ + \frac{X_{\alfa}f}{2f}c_{X_{\beta}}X_{\alfa} + \frac{X_{\beta}f}{2f}c_{X_{\alfa}}X_{\alfa} + c_{X_{\alfa}}(c_{X_{\beta}}X_{\alfa}) +$$
 $$-\left( \frac{X_{\beta}X_{\alfa}f}{2f}X_{\alfa} + \frac{X_{\beta}f}{2f}c_{X_{\alfa}}X_{\alfa} + \frac{X_{\alfa}f}{2f}c_{X_{\beta}}X_{\alfa} + c_{X_{\beta}}(c_{X_{\alfa}}X_{\alfa})\right) +$$
 $$\left. + \frac{(\pi_{\g{T}}T)(X_{\alfa},X_{\beta})f}{2f}X_{\alfa} + \frac{1}{2f}\sum_{i=1}^k(\nabla_{(\pi_{\g{T}}T)(X_{\alfa},X_{\beta})}J_{\theta^i})J_{\theta^i}X_{\alfa} +  \frac{1}{f}\gamma_{(\pi_{\g{T}}T)(X_{\alfa},X_{\beta})}X_{\alfa}, X_{\beta}\right)=$$
 $$= fK+ fg\left(c_{X_{\alfa}}(c_{X_{\beta}}X_{\alfa}) -c_{X_{\beta}}(c_{X_{\alfa}}X_{\alfa}) +c_{(\pi_{\g{H}}T)(X_{\alfa},X_{\beta})}X_{\alfa}, X_{\beta}\right.$$
 $$\left.+\frac{1}{2f}\sum_{i=1}^k(\nabla_{(\pi_{\g{T}}T)(X_{\alfa},X_{\beta})}J_{\theta^i})J_{\theta^i}X_{\alfa} +  \frac{1}{f}\gamma_{(\pi_{\g{T}}T)(X_{\alfa},X_{\beta})}X_{\alfa}, X_{\beta}\right)=$$
 $$= fK+ \frac{C}{f}|\nabla^Hf|^2 + \sum_{\alfa,\beta}C_{\alfa\beta}\nabla^2_{X_{\alfa},X_{\beta}}f$$
 for some constants and $C_{\alfa\beta}$ which are purely algebraic, and therefore can be computed on $G$.
 
 First, necessarily $\sum_{\alfa,\beta}C_{\alfa\beta}\nabla^2_{X_{\alfa},X_{\beta}}f = A\Delta$ for some costant $A$.
 Calling $u=f^{C/A+1}$ is it easy to verify that, since on $G$ $K=0$,
 $$C_1\Delta u = \widetilde{K}u^{C/A+1}.$$
 Since the equation must be covariant by dilations, similarly to the Riemannian and to the CR case, the exponent must be the critical one, $\frac{Q+2}{Q-2}$.
\end{proof}

Similarly to Riemannian and CR geometry, it is interesting to consider the problem of imposing constant curvature.
Define
$$Q(u) = \frac{\int_M C|\nabla^{\g{H}}u|^2 + Ku^2}{\left(\int_M u^{\frac{2Q}{Q-2}}\right)^{\frac{Q-2}{Q}}}$$
and
$$\ci{Y}(M) = \inf_{u>0}Q(u).$$
Then, by following the same proof as Riemannian or CR geometry (see \cite{LP} and \cite{JL1}) the following theorem can be proved.

\begin{teorema}
 Every compact contact manifold of Heisenberg type $M$ modelled on $G$ satisfies $\ci{Y}(M)\le\ci{Y}(G)$, and if the inequality is strict, there exists a conformal charge with constant scalar curvature.
\end{teorema}

We notice that $\ci{Y}(\bm{S}_G)=\ci{Y}(G)$, while we do not now whether in the case of $G$ not of Iwasawa type, there exists a compact manifold such that $\ci{Y}(M)=\ci{Y}(G)$.

\textsc{Claudio Afeltra, University of Montpellier,  
	Institut Montpelliérain  Alexander Grothendieck.
	Place Eug\`ene Bataillon, 
	34090 Montpellier,  France}

\textit{Email address}:  \texttt{claudio.afeltra@umontpellier.it}

\end{document}